\documentclass[11pt]{article}
\usepackage{mathrsfs}
\usepackage{tikz}
\usetikzlibrary{calc}
\usetikzlibrary{decorations.markings,arrows}
\usepackage{subcaption}

\usepackage{amsmath,amsthm,verbatim,amssymb,amsfonts,amscd, graphicx, float}
\usepackage{graphics}
\usepackage{enumitem}
\usepackage{bbm}
\usepackage[margin=1in]{geometry}
\usepackage{thmtools}
\usepackage{thm-restate}
\usepackage{hyperref}
\usepackage[noabbrev]{cleveref}
\Crefname{subsection}{subsection}{subsections}

\newtheorem{theorem}{Theorem}[section]
\newtheorem{corollary}[theorem]{Corollary}
\newtheorem{lemma}[theorem]{Lemma}

\newtheorem{proposition}[theorem]{Proposition}
\newtheorem{remark}[theorem]{Remark}
\newtheorem{definition}[theorem]{Definition}

\newcommand\restr[2]{{
  \left.\kern-\nulldelimiterspace 
  #1 
  \vphantom{\big|} 
  \right|_{#2} 
  }}

\newcommand{\floor}[1]{\left\lfloor #1 \right\rfloor}
\newcommand{\ceil}[1]{\left\lceil #1 \right\rceil}

\newcommand{\inner}[1]{\left\langle{#1}\right\rangle}
\newcommand{\norm}[1]{\left\lVert#1\right\rVert}

\newcommand{\R}{\mathbb{R}} 

\renewcommand{\Pr}{\mathop{\bf Pr\/}}
\newcommand{\E}{\mathop{\mathbb E\/}}

\newcommand{\indicator}{\mathbbm{1}}

\newcommand{\Ent}{\mathop{\bf Ent\/}}

\newcommand{\sn}{{\mathbb{S}^{n-1} }}

\newcommand{\sk}{{\mathcal{S}_k }}

\newcommand{\FR}{{\mathrm{FR}_\gamma^n}}

\definecolor{MyDarkBlue}{rgb}{0,0.08,1}
\definecolor{MyDarkGreen}{rgb}{0.02,0.6,0.02}
\definecolor{MyDarkRed}{rgb}{0.8,0.02,0.02}
\definecolor{MyDarkOrange}{rgb}{0.40,0.2,0.02}
\definecolor{MyPurple}{RGB}{111,0,255}
\definecolor{MyRed}{rgb}{1.0,0.0,0.0}
\definecolor{MyGold}{rgb}{0.75,0.6,0.12}
\definecolor{MyDarkgray}{rgb}{0.66, 0.66, 0.66}

\date{}

\begin{document}
\title{\textbf{Density Frankl–Rödl on the Sphere}}
\author{
Venkatesan Guruswami\thanks{{Departments of EECS and Mathematics, and the Simons Institute for the Theory of Computing, UC Berkeley, Berkeley, CA, 94709, USA. Email: \url{venkatg@berkeley.edu}. Research supported by a Simons Investigator Award and NSF grants CCF-2210823 and CCF-2228287.}}
\and 
Shilun Li\thanks{Department of Mathematics, UC Berkeley, Berkeley, CA, 94709, USA. Email: \url{shilun@berkeley.edu}.} 
}
\maketitle
\begin{abstract}
We establish a density variant of the Frankl–Rödl theorem on the sphere $\mathbb{S}^{n-1}$, which concerns avoiding pairs of vectors with a specific distance, or equivalently, a prescribed inner product. In particular, we establish lower bounds on the probability that a randomly chosen pair of such vectors lies entirely within a measurable subset $A \subseteq \mathbb{S}^{n-1}$ of sufficiently large measure. Additionally, we prove a density version of spherical avoidance problems, which generalize from pairwise avoidance to broader configurations with prescribed pairwise inner products. Our framework encompasses a class of configurations we call \textit{inductive configurations}, which include simplices with any prescribed inner product $-1 < r < 1$. As a consequence of our density statement, we show that all non-antipodal inductive configurations are sphere Ramsey (for measurable colorings).
\vspace{0.2in}
\end{abstract}

\section{Introduction}

The Frankl–R\"{o}dl theorem~\cite{frankl1987forbidden} is a foundational result in extremal combinatorics and theoretical computer science. It states that for any fixed \( 0 < \gamma < 1 \), assuming \( (1 - \gamma)n \) is even, if a set \( A \subseteq \{-1,1\}^n \) contains no pair of distinct points at Hamming distance exactly \( (1 - \gamma)n \), then the fractional size of \( A \) must be exponentially small. That is, there exists a constant \( \epsilon = \epsilon(\gamma) < 1 \) such that
\[
|A|/2^{n} \leq \epsilon^n.
\]
The corresponding Frankl–R\"{o}dl graph \( \FR \) is defined on vertex set \( \{-1,1\}^n \) with edges between points at Hamming distance \( (1 - \gamma)n \), with the Frankl-R\"{o}dl theorem bounding the independence number of this graph. This theorem has found broad applications in theoretical computer science, particularly in the analysis of hardness of approximation. The graph \( \FR \) has been used to construct integrality gap instances for problems such as 3-Coloring~\cite{kleinberg1998lovasz, karger1998approximate, charikar2002semidefinite, arora2011new, kauers2014hypercontractive} and Vertex-Cover~\cite{kleinberg1998lovasz, charikar2002semidefinite, arora2006proving, georgiou2008vertex, georgiou2010integrality, kauers2014hypercontractive}.

\smallskip
The Frankl–R\"{o}dl theorem extends naturally from pairs of points to general forbidden configurations. A configuration of \( k \) vertices is specified by a set of pairwise distances (or equivalently, inner products). Frankl and R\"{o}dl~\cite{frankl1987forbidden} showed that for any fixed \( k \), any subset \( A \subseteq \{-1,1\}^{4n} \) that avoids \( r \) pairwise orthogonal vectors must also have exponentially small fractional size. This result has been applied to show an \( \Omega(\log n) \) integrality gap for the SDP relaxation of Min-Multicut~\cite{agarwal2005log}.

\smallskip
While the original theorem asserts that for any \( A \subseteq \{-1,1\}^n \) of size \( |A|/2^n \geq \epsilon(\gamma)^n \),
\[
\Pr_{\substack{x, y \in \{-1,1\}^n \\ x \cdot y = (2\gamma - 1)n}}(x \in A, y \in A) > 0,
\]
it is natural to ask whether one can give a quantitative lower bound on this probability in terms of the density \( |A| / 2^n \). Benabbas, Hatami, and Magen~\cite{benabbas2012isoperimetric} answered this by proving a density version of the Frankl–R\"{o}dl theorem: for any \( 0 < \alpha < 1 \), if \( |A| / 2^n \geq \alpha \), then
\[
\Pr_{\substack{x, y \in \{-1,1\}^n \\ x \cdot y = (2\gamma - 1)n}}(x \in A, y \in A) \geq 2\left(\alpha/2 \right)^{\frac{2}{1 - |2\gamma - 1|}} - o(1).
\]
This result is based on the reverse hypercontractivity of the Bonami–Beckner semigroup~\cite{mossel2006non}, and was used to prove integrality gaps for Vertex-Cover. Later Kauers et. al \cite{kauers2014hypercontractive} showed that the proof can be conducted in the sum-of-squares (SOS) proof system. In particular, they showed that for any $0<\gamma\leq 1/4$, the SOS/Lasserre SDP hierarchy at degree $4\ceil{\frac{1}{4\gamma}}$ certifies that the maximum independent set in $\FR$ has fractional size $o(1)$. This implies that a degree-4 algorithm from the SOS hierarchy can certify that the $\FR$ SDP integrality gap instances for 3-Coloring have chromatic number $\omega(1)$, and that a degree-$\ceil{1/\gamma}$ SOS algorithm can certify the $\FR$ SDP integrality gap instances for Min-Vertex-Cover has minimum vertex cover $>1-o(1)$.

\smallskip In the continuous setting, the sphere \( \mathbb{S}^{n-1} \) provides a natural high-dimensional geometric analogue. This motivates the study of spherical avoidance problems: how dense can a subset \( A \subseteq \mathbb{S}^{n-1} \) be while avoiding a fixed configuration of pairwise inner products? Let \( \Delta_k(n, r) \) denote the \( k \)-simplex in \( \mathbb{S}^{n-1} \) with pairwise inner product \( r \). Witsenhausen's problem~\cite{witsenhausen1974spherical} asks for the maximum density of a measurable set \( A \subseteq \mathbb{S}^{n-1} \) avoiding orthogonal pairs, i.e., \( \Delta_2(n, 0) \). Frankl and Wilson~\cite{frankl1981intersection} gave the first exponentially decreasing upper bound \( \sigma(A) \leq (1+o(1))(1.13)^{-n} \), where \( \sigma \) denotes the uniform surface measure on \( \mathbb{S}^{n-1} \). Kalai~\cite{kalai2009large} conjectured that the extremal set consists of two opposite caps of geodesic radius \( \pi/4 \); this \emph{Double Cap Conjecture} implies new lower bounds for the measurable chromatic number of \( \mathbb{R}^n \)~\cite{decorte2016spherical}. 

\smallskip
Regev and Klartag~\cite{regev2011quantum} established a density version of the Frankl–Rödl theorem on the sphere for pairs of orthogonal vectors:
$$
\Pr_{\substack{x, y \in \sn \\ x \cdot y = 0}}(x \in A,\ y \in A) \geq 0.9\,\sigma(A)^2,
$$
valid for any measurable set $A \subseteq \sn$ with $\sigma(A) \geq C\exp(-cn^{1/3})$, where $C$ and $c$ are universal constants. This result was a key component in their proof of an $\Omega(n^{1/3})$ lower bound on the classical communication complexity of the Vector in Subspace Problem (VSP). As a major consequence, they resolved a long-standing open question posed by Raz~\cite{raz1999exponential}, demonstrating that quantum one-way communication is indeed exponentially stronger than classical two-way communication.

\smallskip
For configurations \( \Delta_k(n, r) \) with \( r > 0 \), Castro-Silva et al.~\cite{castro2022recursive} proved that for any \( k \geq 2 \), there exists \( c = c(k, r) < 1 \) such that any \( A \subseteq \mathbb{S}^{n-1} \) avoiding \( \Delta_k(n, r) \) satisfies \( \sigma(A) \leq (c + o(1))^n \). A configuration is sphere Ramsey if any $c$-coloring of the sphere contains a monochromatic congruent copy of the configuration. Castro-Silva et al.'s result implied that all simplicial configurations \( \Delta_k(n, r) \) with $r>0,\ k\geq 2$ are sphere Ramsey. On the other hand, Matoušek and R\"{o}dl~\cite{matouvsek1995ramsey} showed that any configuration \( P \) with circumradius less than 1 is sphere Ramsey. 

\smallskip
This sphere Ramsey theorem was recently used by Brakensiek, Guruswami, and Sandeep~\cite{brakensiek2023sdps} to demonstrate integrality gaps for the basic SDP relaxation of certain promise CSPs: namely Boolean symmetric promise CSPs defined by a single predicate pair that lack Majority or Alternate-Threshold (AT) polymorphisms of all odd arities. Via Raghavendra's general connection tightly linking SDP integrality gaps to Unique-Games hardness~\cite{raghavendra2008optimal}, this enabled BGS to conclude that such promise CSPs do not admit a robust satisfiability algorithm (in the sense of Zwick~\cite{zwick1998finding}) under the Unique Games conjecture. Complementing this hardness result, BGS gave a robust satisfiability algorithm for Boolean promise CSPs that admit Majority polymorphisms of all odd arities or AT polymorphisms of all odd arities---the algorithm applied in the general promise CSP setting that could have multiple predicate pairs. Together these results led to a dichotomy theorem with respect to robust satisfiability for Boolean symmetric promise CSPs defined by a single predicate. Towards extending their hardness result to Boolean symmetric PCSPs that could include multiple predicate pairs, BGS posed the following problem: can one show a density version of spherical avoidance problems for the configuration $(x_1,\cdots,x_b,-x_{b+1},\cdots, -x_k)$ where $(x_1,\cdots,x_b,x_{b+1},\cdots, x_k)\in\Delta_{k}(n,r)$. Namely, obtain a non-trivial lower bound on the quantity
$$\Pr_{(x_1,\cdots,x_b,x_{b+1},x_k)\in\Delta_{k}(n,r)}\left(x_1\in A,\cdots,x_b\in A,-x_{b+1}\in A,\cdots, -x_k\in A\right).$$
in terms of $\sigma(A)$. While we do not directly resolve the exact configuration posed by BGS, our result applies to a broad class of configurations that includes closely related structures.

\smallskip
In this paper, we show the following result:
\begin{theorem}
    Fix $k,r$ such that \( -\frac{1}{k-1} < r < 1 \), there exists constant \( C = C(k,r),\ \epsilon=\epsilon(k,r) \) such that
\[
\Pr_{x_1,\dots,x_k \in \Delta_k(n,r)}(x_i \in A\ \forall i) \geq \Omega_{k,r}\left(\sigma(A)^C\right)
\]
for all measurable \( A \subseteq \mathbb{S}^{n-1} \) with \( \sigma(A) \geq \omega_{k,r}(n^{-\epsilon}) \).
\end{theorem}

Our result generalizes the work of Castro-Silva et al. to simplices with any \( -\frac{1}{k-1} < r < 1 \), and more broadly, to a class of \emph{inductive configurations} defined in \Cref{sec: sphere configurations} (see \Cref{thm: density R configuration}). In particular, we show that all non-antipodal inductive configurations are sphere Ramsey.

\smallskip
Inspired by the techniques of Benabbas et al.~\cite{benabbas2012isoperimetric}, who used reverse hypercontractivity of the Bonami–Beckner semigroup to show the density variant of Frankl–R\"{o}dl on $\{-1,1\}^n$, we develop analogous tools on the sphere. Specifically, we prove reverse hypercontractivity of the operator
\[
A_r f(x) := \E_{\substack{y \in \mathbb{S}^{n-1} \\ x \cdot y = r}} [f(y)],
\]
which enables us to derive density versions of spherical avoidance for inductive configurations. In \Cref{sec: spherical harmonics}, we analyze the eigen-decomposition of \( A_r \); in \Cref{sec: reverse hypercontrative}, we relate \( A_r \) to the Poisson Markov semigroup \( P_t \) and establish reverse hypercontractive inequalities. Finally, in \Cref{sec: sphere configurations}, we prove our main results on density versions of the Frankl–R\"{o}dl theorem on the sphere and density spherical Ramsey statements for inductive configurations.

\section{Eigen-decomposition via Spherical Harmonics}\label{sec: spherical harmonics}
In order to analyze the eigen-decomposition of $A_r$, we first introduce some fundamental notions from the theory of spherical harmonics. This decomposition of $A_r$ plays a central role in our analysis of reverse hypercontractivity. For a comprehensive background, we refer to the standard references \cite{muller2006spherical, stein1971introduction} and notes such as \cite{gallier2009notes}.

\subsection{Spherical Harmonics}

For each integer $n \geq 2$, let $\mathbb{S}^{n-1} \subseteq \mathbb{R}^n$ denote the unit $n$-sphere, defined by
\[
\mathbb{S}^{n-1} = \{ x \in \mathbb{R}^n : x \cdot x = 1\}.
\]
Given $p > 0$ and a measurable function $f: \mathbb{S}^{n-1} \to \mathbb{R}$, we define its $L^p$ norm by
\[
\|f\|_{L^p(\mathbb{S}^{n-1})} := \left( \int_{\mathbb{S}^{n-1}} |f(x)|^p \, d\sigma(x) \right)^{1/p},
\]
where $\sigma$ denotes the uniform surface measure on $\mathbb{S}^{n-1}$, normalized such that $\int_{\mathbb{S}^{n-1}} d\sigma(x) = 1$. It follows that for $p \geq q > 0$, we have $L^p(\mathbb{S}^{n-1}) \subseteq L^q(\mathbb{S}^{n-1})$.

For $q < 0$ and a strictly positive function $f > 0$, we extend the definition by
\[
\|f\|_q := \left( \int_{\mathbb{S}^{n-1}} f(x)^q \, d\sigma(x) \right)^{1/q}.
\]
Note that $\|f\|_q$ is well-defined for $f > 0$ and $q < 0$, since $f^q > 0$ and $1/q < 0$ reverses the inequality direction (so this is a quasi-norm satisfying $\|cf\|_q = c\|f\|_q$ for $c > 0$).

\smallskip
We focus in particular on $L^2(\mathbb{S}^{n-1})$, the space of square-integrable functions. It admits an orthonormal basis consisting of spherical harmonics $\{Y_{k,\ell}\}$, indexed by integers $k \geq 0$ (the degree) and $1 \leq \ell \leq c_{k,n}$, where
\[
c_{k,n} = \binom{k+n-2}{k} + \binom{k+n-3}{k-1}.
\]
Each $Y_{k,\ell}$ is the restriction to $\mathbb{S}^{n-1}$ of a harmonic, homogeneous polynomial of degree $k$ on $\mathbb{R}^n$. We denote by $\mathcal{S}_k \subseteq L^2(\mathbb{S}^{n-1})$ the finite-dimensional subspace spanned by spherical harmonics of degree $k$:
\[
\mathcal{S}_k := \operatorname{span} \{ Y_{k,1}, \dotsc, Y_{k,c_{k,n}} \}.
\]
Notably, $\mathcal{S}_0$ consists of the constant functions. Rotations $U \in \mathrm{SO}(n)$ act on $f \in L^2(\mathbb{S}^{n-1})$ via
\[
(Uf)(x) := f(U^{-1}x).
\]
Each space $\mathcal{S}_k$ is invariant under the action of $\mathrm{SO}(n)$, and furnishes an irreducible representation, i.e. any subspace $V\subseteq\sk$ that is invariant under rotations must either be $\{0\}$ or $\sk$. Moreover, for $n \geq 3$, the spaces $\mathcal{S}_k$ for different $k$ have different dimensions therefore are inequivalent as representations. Throughout the remainder of this paper, we assume $n \geq 3$.

\smallskip
For $r \in (0,1]$ (setting $t = -\log r \geq 0$), define an operator $A_r : L^2(\mathbb{S}^{n-1}) \to L^2(\mathbb{S}^{n-1})$ by
\[
(A_r f)(x) := \int_{\mathbb{S}^{n-1}} f(y) \, d\sigma_{x,r}(y) = \E_{\substack{y\in \sn\\\  x \cdot y = r}} [f(y)],
\]
where $\sigma_{x,r}$ is the uniform probability measure on the $(n-2)$-subsphere $S_{x,r}:=\{ y \in \mathbb{S}^{n-1} : x \cdot y = r \}$.

The operator $A_r$ is easily seen to commute with rotations via
$$(A_r U f)(x)=\E_{\substack{y\in \sn\\\  x \cdot y = r}} [f(U^{-1}y)]=\E_{\substack{z\in \sn\\\  x \cdot Uz = r}} [f(z)]=\E_{\substack{z\in \sn\\\  U^{-1}x \cdot z = r}} [f(z)]=(UA_rf)(x)$$
for any $U\in \mathrm{SO}(n)$.
It is also self-adjoint by
$$\inner{A_r f,g}=\E_{x\in \sn}[g(x)A_rf(x)]=\E_{x\in \sn}\left[\E_{y:x\cdot y=r}[g(x)f(y)]\right]=\E_{y\in \sn}\left[\E_{x:x\cdot y=r}[g(x)f(y)]\right]=\inner{f,A_r g}.$$
for all $f,g \in L^2(\mathbb{S}^{n-1})$.

\smallskip
Using standard techniques from representation theory, we establish that the spaces $\mathcal{S}_k$ are eigenspaces of $A_r$:

\begin{lemma}\label{lem: eigenspace}
For any $k \geq 0$ and $Y_k \in \mathcal{S}_k$, we have
\[
A_r(Y_k) = \mu_{k,r} Y_k,
\]
where $\mu_{k,r} = G_k(r)$, with $G_k : [-1,1] \to \mathbb{R}$ denoting the degree-$k$ Gegenbauer polynomial (see,e.g., \cite{muller2006spherical}), given by
\[
G_k(r) = \mathbb{E} \left[ \left( r + i X_1 \sqrt{1 - r^2} \right)^k \right],
\]
where $X = (X_1, \dotsc, X_{n-1})$ is uniformly distributed on $\mathbb{S}^{n-2}$ and $X_1$ denotes any fixed coordinate.
\end{lemma}

\begin{proof}
By Schur’s Lemma (see e.g. \cite{serre1977linear}), any linear operator commuting with the group action must act as a scalar on each irreducible representation. Since $A_r$ commutes with rotations and $\mathcal{S}_k$ are inequivalent irreducible representations for different $k$, the restriction of $A_r$ to $\mathcal{S}_k$ must be a scalar multiple of the identity.

To identify the eigenvalue, consider the function $f_k(x) := G_k(x \cdot v)$ for some fixed $v \in \mathbb{S}^{n-1}$. It is known that $f_k \in \mathcal{S}_k$ \cite{muller2006spherical}. Then
\[
(A_r f_k)(v) = G_k(r), \quad \text{and} \quad f_k(v) = G_k(1) = 1,
\]
showing that $\mu_{k,r} = G_k(r)$.
\end{proof}

\subsection{Eigenvalue Estimates}
We now estimate the eigenvalues $\mu_{k,r}$.

\begin{lemma}\label{lem: mu estimate}
The eigenvalues $\mu_{k,r}$ satisfy
\[
|\mu_{k,r} - r^k| = O_r\!\left(\frac{1}{n}\right),
\]
as $n \to \infty$, uniformly in $k$.
\end{lemma}
\begin{proof}
Since $X_1$ has an even density, odd moments vanish. For even moments, \cite{regev2011quantum} Lemma~5.5 gives $\E[X_1^{2a}] \leq (2a/(n-4))^a$.

\textbf{Case $k \leq n/10$.}
From $G_k(r) = \E[(r+iX_1\sqrt{1-r^2})^k]$ and vanishing of odd moments,
\[
\mu_{k,r} - r^k
= \sum_{a=1}^{\floor{k/2}}(-1)^a\binom{k}{2a}(1-r^2)^a r^{k-2a}\,\E[X_1^{2a}].
\]
Applying $\binom{k}{2a}\leq k^{2a}/(2a)!$, the moment bound, and $k^{2a}|r|^{k-2a} \leq (2a/(e|\!\log|r||))^{2a}$ (maximum of $j\mapsto j^{2a}|r|^j$ at $j=2a/|\!\log|r||$), then $(2a)!\geq(2a/e)^{2a}$:
\[
|\mu_{k,r}-r^k|
\;\leq\; \sum_{a=1}^{\infty}
\left(\frac{2(1-r^2)}{(n-4)\log^2(1/|r|)}\right)^{\!a} a^a
\;=:\; \sum_{a=1}^\infty v^a\, a^a,
\]
where $v = 2(1-r^2)/((n-4)\log^2(1/|r|)) = O_r(1/n)$.
The series $\sum a^a v^a$ converges for $|v|<1/e$ (ratio test: consecutive ratio $\to e|v|$) and equals $O(v) = O_r(1/n)$.

\textbf{Case $k \geq n/10$.}
By Markov's inequality and the moment bound,
$\Pr(|X_1|\geq\tfrac12) \leq (4\lceil n/20\rceil/(n-4))^{n/20} = O(1/n)$.
Hence, using $|G_k(r)|\leq\E[(r^2+(1-r^2)X_1^2)^{k/2}]$,
\begin{align*}
|\mu_{k,r}-r^k|
&\leq \Pr\!\left(|X_1|<\tfrac12\right)\!\left(\tfrac{3r^2+1}{4}\right)^{k/2}
  +\Pr\!\left(|X_1|\geq\tfrac12\right)+|r|^{n/10}\\
&\leq \left(\tfrac{3r^2+1}{4}\right)^{n/20}+O\!\left(\tfrac1n\right)+O_r\!\left(\tfrac1n\right)
= O_r\!\left(\tfrac1n\right).\qedhere
\end{align*}
\end{proof}

\section{Reverse Hypercontractivity of \texorpdfstring{$A_r$}{TEXT}}\label{sec: reverse hypercontrative}
Hypercontractive inequalities play a central role in modern analysis, probability, and theoretical computer science, quantifying how semigroups ``smooth out'' irregularities over time by contracting \(L^q\) norms to \(L^p\) norms. A semigroup \( (T_t)_{t \geq 0} \) is hypercontractive if for all real-valued functions \(f\),
\[
\|T_t f\|_p \leq \|f\|_q,
\]
for \(1 < q < p < \infty\), whenever \(t \ge \tau(p,q)\). These inequalities are deeply connected to logarithmic Sobolev inequalities \cite{gross1975logarithmic} and underlie results on quantum information theory \cite{montanaro2012some}, diffusion processes \cite{bakry2006diffusions}, and mixing times \cite{bobkov2003modified}.

On the Boolean hypercube \(\{0,1\}^n\), a canonical example is the \emph{Bonami–Beckner semigroup} \(T_\rho\), where \(0 \leq \rho \leq 1\), defined as
\[
T_\rho f(x) = \mathbb{E}[f(y)], \quad \text{where } y \sim_\rho x,
\]
and \(y\) is obtained by flipping each bit of \(x\) independently with probability \(\frac{1 - \rho}{2}\). The operator satisfies the sharp hypercontractive inequality \cite{bonami1970etude, nelson1973free, beckner1975inequalities, gross1975logarithmic}:
\[
\|T_\rho f\|_p \le \|f\|_q \quad \text{if } \rho \le \sqrt{\frac{q-1}{p-1}}
\]
for $1<q<p<\infty$. 
This inequality underpins many foundational results in the analysis of Boolean functions, such as the KKL theorem \cite{mossel2013reverse}, the invariance principle \cite{mossel2013reverse}.

In contrast, reverse hypercontractivity, introduced by Mossel, Oleszkiewicz, and Sen \cite{mossel2013reverse}, captures a complementary ``anti-smoothing'' phenomenon. A semigroup \( (T_t)_{t \geq 0} \) is reverse hypercontractive if for all non-negative functions \(f\),
\[
\|T_t f\|_q \ge \|f\|_p,
\]
for \(0 < q < p < 1\), whenever \(t \ge \tau(p,q)\). This inequality lower-bounds the dispersion of mass under the semigroup and connects to reverse log-Sobolev inequalities \cite{mossel2013reverse}, Gaussian isoperimetry \cite{mossel2015robust}, and tail bounds in correlated settings.

On the Boolean hypercube, the Bonami–Beckner semigroup \(T_\rho\) also satisfies reverse hypercontractive inequalities \cite{mossel2013reverse}: for all non-negative $f : \{0,1\}^n \to \mathbb{R}_{\geq 0}$,
\[
\|T_\rho f\|_q \ge \|f\|_p\quad\text{if } \rho < \sqrt{\frac{1 - p}{1 - q}} 
\]
for all $0<q<p<1$. 

Reverse hypercontractivity has enabled several significant applications on the hypercube. It was used to prove the ``It Ain’t Over Till It’s Over" conjecture \cite{khot2002power} from social choice theory \cite{mossel2005noise}, and dimension-free Gaussian isoperimetric inequalities \cite{mossel2015robust}.

The operator $A_r$ serves as the analogue of the Bonami–Beckner operator on the sphere. In this section, we demonstrate that $A_r$ is close to the Poisson Markov semigroup $P_t$ in the $L^2$ norm. Using the logarithmic Sobolev inequality for $P_t$, we then establish that $P_t$ satisfies reverse hypercontractivity inequalities, which in turn imply reverse hypercontractivity inequalities for $A_r$.

\subsection{Poisson Markov Semigroup}
\begin{definition}
A family \( (P_t)_{t \geq 0} \) of operators on real-valued measurable functions on \( \sn \) is called a \emph{Markov semigroup} if it satisfies the following properties:
\begin{enumerate}[label=(\roman*)]
    \item For each \( t \geq 0 \), \( P_t \) is a linear operator mapping bounded measurable functions to bounded measurable functions.
    \item \( P_t(1) = 1 \), where \( 1 \) denotes the constant function on \( \sn \) (mass conservation).
    \item If \( f \geq 0 \), then \( P_t f \geq 0 \) (positivity preservation).
    \item \( P_0 = \mathrm{Id} \), the identity operator.
    \item For any \( s, t \geq 0 \), \( P_{s+t} = P_s \circ P_t \) (semigroup property).
    \item For each \( f \), the map \( t \mapsto P_t f \) is continuous.
\end{enumerate}
\end{definition}

Operators satisfying properties (i)–(iii) are called \emph{Markov operators}. For any function \( f \in L^2(\sn) \), we can express it in terms of the spherical harmonic basis as
\[
f(x) = \sum_{k,l} \widehat{f_{k,l}} Y_{k,l}(x),
\]
where \( \widehat{f_{k,l}} = \langle f, Y_{k,l} \rangle \). Define the \emph{Poisson semigroup}\footnote{In the literature, e.g., \cite{beckner1992sobolev}, the Poisson semigroup is typically defined by \( P_r f(x) = \sum_{k,l} r^k \widehat{f_{k,l}} Y_{k,l}(x) \) with multiplicative semigroup property \( P_{rs} = P_r \circ P_s \). We adopt a slightly modified definition here to match the additive semigroup convention of Markov semigroups.} \( (P_t)_{t \geq 0} : L^2(\sn) \to L^2(\sn) \) by
\[
P_t f(x) = \sum_{k,l} e^{-kt} \widehat{f_{k,l}} Y_{k,l}(x).
\]
As shown in \cite{beckner1992sobolev}, this can equivalently be written in terms of the Poisson kernel
\[
K_r(x, y) = \frac{1 - r^2}{|x - r y|^n}, \quad (-1 \leq r \leq 1),
\]
so that
\[
P_t f(x) = \int_{\sn} K_{e^{-t}}(x, y) f(y)\, d\sigma(y) = \int_{\sn} K_r(x, y) f(y)\, d\sigma(y) \quad (r = e^{-t}).
\]
It is straightforward to verify that \( P_t \) satisfies the Markov semigroup properties. 

We now show that \( A_r \) is close to \( P_t \) in the \( L^2 \)-norm (where $t = -\log r$).

\begin{lemma}
\label{lem: equivalence of A and P}
    For any \( f \in L^2(\sn) \),
    \[
    \| A_r f - P_t f \|_2 = O_r\left(\frac{1}{n}\right) \| f \|_2
    \]
    as \( n \to \infty \).
\end{lemma}

\begin{proof}
    Expressing \( f \) in the spherical harmonic basis as \( f(x) = \sum_{k,l} \widehat{f_{k,l}} Y_{k,l}(x) \), we have
    \[
    \| A_r f - P_t f \|_2^2 = \left\| \sum_{k,l} (\mu_{k,r} - r^k) \widehat{f_{k,l}} Y_{k,l} \right\|_2^2 = \sum_{k,l} |\mu_{k,r} - r^k|^2 \left| \widehat{f_{k,l}}\right|^2.
    \]
    By \Cref{lem: mu estimate},
    \[
    |\mu_{k,r} - r^k|^2 \leq O_r(n^{-2})
    \]
    where we can conclude
    \[
    \| A_r f - P_t f \|_2^2 \leq O_r\left(n^{-2}\right) \sum_{k,l} \left| \widehat{f_{k,l}} \right|^2 = O_r\left(n^{-2}\right) \| f \|_2^2\ . \qedhere
    \]
\end{proof}

\subsection{Reverse Hypercontractivity}

Mossel, Oleszkiewicz, and Sen~\cite{mossel2013reverse} showed that for Markov semigroups, reverse hypercontractivity inequalities follow from log-Sobolev inequalities. We will use this fact to establish reverse hypercontractivity for the Poisson semigroup \( P_t \). For any non-negative function \( f \geq 0 \), define its entropy by
\[
\Ent(f) = \E[f \log f] - \E[f] \cdot \log \E[f].
\]
For a Markov semigroup \( (P_t)_{t \geq 0} \) with generator \( L \), the Dirichlet form is defined as
\[
\mathcal{E}(f, g) = \E[f L g] = \E[g L f] = \mathcal{E}(g, f) = -\left. \frac{d}{dt} \E[f P_t g] \right|_{t = 0}.
\]

\begin{lemma}[\cite{beckner1992sobolev}, Theorem 1]
\label{lem: log-Sobolev of P}
The Poisson semigroup \( (P_t)_{t \geq 0} \) satisfies the log-Sobolev inequality:
\[
\Ent(f^2) \leq C \mathcal{E}(f, f),
\]
with sharp constant \( C = 2 \).
\end{lemma}
To see $C=2$ is tight, take \( f = 1 + \varepsilon Y_1 \) for a normalized degree-one spherical harmonic \( Y_1 \); then \( \Ent(f^2) = 2\varepsilon^2 + o(\varepsilon^2) \) while \( \mathcal{E}(f,f) = \varepsilon^2 \), so \( C \geq 2 \). The upper bound \( C \leq 2 \) follows from \cite{beckner1992sobolev}, since \( \mathcal{E}(f,f) = \sum_{k,\ell} k|\widehat{f_{k,\ell}}|^2 \geq \sum_{k,\ell} \Delta_n(k)|\widehat{f_{k,\ell}}|^2 \) (as \( \Delta_n(k) \leq k \)) and Beckner's inequality gives \( \Ent(f^2) \leq 2\sum_{k,\ell}\Delta_n(k)|\widehat{f_{k,\ell}}|^2 \).

\begin{lemma}[\cite{mossel2013reverse}, Theorem 1.10]
\label{lem: reverse hypercontractivity of P}
If a Markov semigroup \( (P_t)_{t \geq 0} \) satisfies the log-Sobolev inequality with constant \( C \), then for all \( q < p < 1 \) and every non-negative function \( f: \sn \to \R \), the following inequality holds for all \( t \geq \frac{C}{4} \log \frac{1 - q}{1 - p} \):
\[
\|P_t f\|_q \geq \|f\|_p.
\]
\end{lemma}

We now deduce a reverse hypercontractivity inequality for the operator \( A_r \), leveraging its proximity to \( P_t \) in the \( L^2 \) norm.

\begin{theorem}
\label{thm: two function reverse hypercontractivity}
For any $r \in (-1,1)$ and any non-negative functions \( f, g \in L^2(\sn) \), the following reverse hypercontractivity inequality holds:
\[
\E_{x \cdot y = r}[f(x) g(y)] = \langle f, A_r g \rangle \geq \|f\|_p \|g\|_p - O_r\left(\frac{1}{n}\right) \|f\|_2 \|g\|_2,
\]
as \( n \to \infty \), for all \( 0 < p \leq 1 - |r| \).
\end{theorem}

\begin{proof}
If $r < 0$, set $\tilde{g}(y) = g(-y)$; since $\sigma$ is invariant under $y \mapsto -y$ we have $\|\tilde{g}\|_s = \|g\|_s$ for all $s$, and $\E_{x \cdot y = r}[f(x)g(y)] = \E_{x \cdot y = |r|}[f(x)\tilde{g}(y)]$, reducing to the case $r > 0$.

For $r > 0$, set $t = -\log r$. By \Cref{lem: equivalence of A and P}, we have
\[
\E_{x\cdot y=r}[f(x)g(y)]=\langle f, A_r g \rangle \geq \langle f, P_t g \rangle - O_r\left(\frac{1}{n}\right) \|f\|_2 \|g\|_2.
\]
Let \( p' = \frac{p}{p - 1} \) be the Hölder conjugate of \( p \), so that \( 1/p + 1/p' = 1 \). Applying the reverse Hölder inequality (see~\cite{hardy1952inequalities}, Theorem 13), we obtain
\[
\langle f, P_t g \rangle = \|f P_t g\|_1 \geq \|f\|_p \|P_t g\|_{p'}.
\]
From \Cref{lem: log-Sobolev of P}, the semigroup \( (P_t)_{t \geq 0} \) satisfies a log-Sobolev inequality with constant \( C = 2 \). Thus, by \Cref{lem: reverse hypercontractivity of P},
\[
\|P_t g\|_{p'} \geq \|g\|_p
\]
provided that \( t \geq \frac{C}{4} \log \frac{1 - p'}{1 - p} \). Since \( p' = p/(p-1) \), we have \( 1 - p' = 1/(1-p) \), so
\[
\frac{C}{4} \log \frac{1 - p'}{1 - p} = \frac{2}{4} \log \frac{1}{(1-p)^2} = -\log(1-p).
\]
Thus the condition is \( t \geq -\log(1-p) \), equivalently \( r = e^{-t} \leq 1-p \), i.e.\ \( p \leq 1 - r \). Therefore, the inequality holds for all \( p \leq 1 - |r| \), completing the proof.
\end{proof}

\section{Density Sphere Avoidance}\label{sec: sphere configurations}
In this section, we present our main results: the density version of the Frankl–R\"{o}dl theorem on the sphere (\Cref{thm: two set probability}) and the density sphere avoidance result for inductive configurations (\Cref{thm: density R configuration}).

\subsection{Inductive Configurations}
Let \( (v_1,\cdots,v_k) \) be a configuration on \( \sn \), where \( v_i \in \sn \) are distinct vectors satisfying \( \inner{v_i, v_j} = r_i \) for \( i < j \). Specifically, writing $v_i$ as column vectors, we consider configurations whose covariance matrix has the form
\[
R(r_1,\cdots,r_{k-1}) := (v_1,\cdots,v_k)^T(v_1,\cdots,v_k) =
\begin{pmatrix}
1 & r_1 & r_1 & \cdots & r_1 & r_1 \\
r_1 & 1 & r_2 & \cdots & r_2 & r_2 \\
r_1 & r_2 & 1 & \cdots & r_3 & r_3 \\
\vdots & \vdots & \vdots & \ddots & \vdots & \vdots \\
r_1 & r_2 & r_3 & \cdots & 1 & r_{k-1} \\
r_1 & r_2 & r_3 & \cdots & r_{k-1} & 1
\end{pmatrix}.
\]
Let \( \Delta(n, R) = \{(x_1, \cdots, x_k) : (x_1, \cdots, x_k)^T(x_1, \cdots, x_k) = R\} \subseteq (\sn)^k \) denote the set of tuples that are congruent to the configuration \( (v_1, \cdots, v_k) \) defined by \( R \). We refer to such configurations as \emph{inductive configurations}, as they can be constructed recursively as follows:

Start with any vector \( v_1 \) and a choice of \( -1 < r_1 < 1 \). Choose \( v_2 \) on the \((n-2)\)-subsphere \( S_{v_1, r_1} = \{x \in \sn : \inner{x, v_1} = r_1\} \). Next, select \( r_2 \) such that the intersection \( S_{v_1, r_1} \cap S_{v_2, r_2} \) is nonempty, and choose \( v_3 \) on this \((n-3)\)-subsphere. Proceed inductively: for each \( v_i \), choose it on the \((n-i)\)-subsphere \( \bigcap_{j < i} S_{v_j, r_j} \).
\\

Throughout the paper, we exclude a special case of inductive configuration--configurations where the length of $v_{k}-v_{k-1}$ is the diameter of $\bigcap_{j<k-1}S_{v_j,r_j}$---for reasons that will become clear in the later proof. Notably, both the length of \( v_k - v_{k-1} \) and the diameter of \( \bigcap_{j < k-1} S_{v_j, r_j} \) are independent of the ambient dimension \( n \), so this is a constraint on the configuration \( R \) itself.
\\

We will show that, fixing any inductive configuration \( R \) (excluding the special case), any set \( A \subseteq \sn \) of constant spherical measure will, as \( n \to \infty \), contain many congruent copies of \( R \). Moreover, we provide an explicit lower bound on the probability that a randomly chosen congruent copy of \( R \) lies entirely within \( A \).

\subsection{Configuration with Pairwise Orthogonal Vectors}

Let \( \Delta_k(n,0) := \Delta(n, I_k) = \{(x_1, \cdots, x_k) \in (\sn)^k : \inner{x_i, x_j} = 0,\ \forall i \neq j\} \) denote the set of all \( k \)-tuples of pairwise orthogonal vectors in \( \sn \). Recall that $S_{x,r}$ denotes the $(n-2)$-subsphere $\{ y \in \mathbb{S}^{n-1} : x \cdot y = r \}$ and $\sigma_{x,r}$ is the uniform probability measure on $S_{x,r}$.
\begin{lemma}[\cite{regev2011quantum}, Theorem 2.2]
\label{lem: orthogonal lemma}
For any measurable set \( A \subseteq \sn \) with \( \sigma(A) > 0 \), the condition
\[
\left| \frac{\sigma_{x,0}(A \cap S_{x,0})}{\sigma(A)} - 1 \right| \leq 0.1
\]
holds for all \( x \in \sn \), except on a subset of measure at most \( C \exp(-c n^{1/3}) \), where \( C, c > 0 \) are universal constants.
\end{lemma}

\begin{theorem}
\label{thm: orthogonal simplex}
Fix an integer \( k \geq 2 \). For any measurable set \( A \subseteq \sn \), the probability that a tuple \( (v_1, \cdots, v_k) \) drawn uniformly at random from \( \Delta_k(n,0) \) lies entirely in \( A \) is at least
\[
\Pr_{(x_1, \cdots, x_k) \in \Delta_k(n,0)}(x_1 \in A, \cdots, x_k \in A) \geq \Omega_k\left(\sigma(A)^k-C_k \exp(-cn^{1/3})\sigma(A)^{k-1}\right),
\]
where \(c > 0 \) is a universal constant and $C_k$ is a constant depending only on $k$.
\end{theorem}

\begin{proof}
We proceed by induction on \( k \). Let \( B \subseteq \sn \) denote the set of all \( x_1 \) for which
\[
\left| \frac{\sigma_{x_1,0}(A \cap S_{x_1,0})}{\sigma(A)} - 1 \right| \leq 0.1.
\]
By \Cref{lem: orthogonal lemma}, we have \( \sigma(B) \geq 1 - C \exp(-c n^{1/3}) \), where \( C, c > 0 \) are universal constants.

When $k=2$,
\begin{align*}
&\Pr_{(x_1, x_2) \in \Delta_2(n,0)}(x_1 \in A, x_2 \in A) \\
\geq{} &\Pr_{(x_1, x_2) \in \Delta_2(n,0)}(x_1 \in A \cap B, x_2 \in A) \\
={} &\Pr(x_1 \in A \cap B) \cdot \Pr(x_2 \in A \mid x_1 \in A \cap B) \\
={} &\Pr(x_1 \in A \cap B) \cdot \E[\sigma_{x_1,0}(A \cap S_{x_1,0}) \mid x_1 \in A \cap B] \\
\geq{} &\left(\sigma(A) - C \exp(-c n^{1/3})\right) \cdot 0.9 \sigma(A) \\
\geq{} &\Omega\left(\sigma(A)^2-C \exp(-cn^{1/3})\sigma(A)\right)
\end{align*}

When $k>2$, assume the claim holds for \( k-1 \). Then:
\begin{align*}
&\Pr_{(x_1, \cdots, x_k) \in \Delta_k(n,0)}(x_1 \in A, \cdots, x_k \in A) \\
\geq{} &\Pr(x_1 \in A \cap B) \cdot \Pr(x_2, \cdots, x_k \in A \mid x_1 \in A \cap B) \\
={} &\Pr(x_1 \in A \cap B) \cdot \Pr(x_2, \cdots, x_k \in A \cap S_{x_1,0} \mid x_1 \in A \cap B).
\end{align*}
By the inductive hypothesis applied within the subsphere \( S_{x_1,0} \), we have:
\begin{align*}
&\Pr(x_2, \cdots, x_k \in A \cap S_{x_1,0} \mid x_1 \in A \cap B) \\
\geq{} &\Omega_k\left((0.9\sigma(A))^{k-1}-C_k\exp(-cn^{1/3})(0.9\sigma(A))^{k-2}\right)\\
\geq{} &\Omega_k\left(\sigma(A)^{k-1}-C_k \exp(-cn^{1/3})\sigma(A)^{k-2}\right).
\end{align*}
Multiplying with \( \Pr(x_1 \in A \cap B) \geq \sigma(A) - C \exp(-c n^{1/3}) \), we obtain:
\begin{align*}
&\Pr(x_1, \cdots, x_k \in A) \\
\geq{} &\Omega_k\left(\sigma(A)^{k-1}-C_k\exp(-cn^{1/3})\sigma(A)^{k-2}\right) \cdot \left(\sigma(A) - C \exp(-c n^{1/3})\right) \\
\geq{} & \Omega_k\left(\sigma(A)^k-C'_{k}\exp(-cn^{1/3})\sigma(A)^{k-1}\right),
\end{align*}
completing the induction.
\end{proof}

\subsection{Main Results}

We first prove a two-set version of the density sphere Frankl-R\"{o}dl theorem which will be later used for induction on inductive configurations.
\begin{theorem}[Density Sphere Frankl-R\"{o}dl]
\label{thm: two set probability}
For any measurable sets \( A, B \subseteq \sn \) and any \( r \in (-1,1) \), we have
\[
\Pr_{x\cdot y = r}(x \in A,\ y \in B) \geq (\sigma(A)\sigma(B))^{1/(1 - |r|)} - O_r\left(\frac{1}{n}\right)\sqrt{\sigma(A)\sigma(B)}.
\]
\end{theorem}

\begin{proof}
The case \( r = 0 \) follows from a stronger version of \Cref{lem: orthogonal lemma}, given by \cite[Theorem 5.1]{regev2011quantum}. Now consider \( r \neq 0 \). Let \( f = \indicator_A \) and \( g = \indicator_B \).
Applying \Cref{thm: two function reverse hypercontractivity} with \( p = 1 - |r| \), we obtain
\begin{align*}
\Pr_{x\cdot y = r}(x \in A,\ y \in B)
&= \E_{x \cdot y = r}[f(x)g(y)] \\
&\geq \|f\|_p \|g\|_p - O_r\left(\frac{1}{n}\right)\|f\|_2 \|g\|_2 \\
&= (\sigma(A)\sigma(B))^{1/(1 - |r|)} - O_r\left(\frac{1}{n}\right)\sqrt{\sigma(A)\sigma(B)}\ . \qedhere
\end{align*}
\end{proof}

Fix an inductive configuration \( R \) and let \( A \subseteq \sn \) be measurable. For any \( x \in \sn \), recall that \( \sigma_{x, r} \) denotes the uniform measure on \( S_{x, r} \). We say that a vector \( x \in A \) is \emph{good} if
\[
\sigma_{x, r}(A \cap S_{x, r}) \geq \frac{1}{2}\sigma(A)^{(1 + |r|)/(1 - |r|)}.
\]
Define the set of good vectors:
\[
A_{\mathrm{good}} = A \cap \left\{x \in \sn : \sigma_{x, r}(A \cap S_{x, r}) \geq \frac{1}{2}\sigma(A)^{(1 + |r|)/(1 - |r|)} \right\}.
\]
Note that $A_{\mathrm{good}} \subseteq A$ by definition. Using \Cref{thm: two set probability}, we can show that many such good vectors must exist.

\begin{proposition}
\label{prop: A_good lower bound}
The set of good vectors satisfies
\[
\sigma(A_{\mathrm{good}}) \geq \frac{1}{2}\sigma(A)^{2/(1 - |r|)} - O_r\left(\frac{1}{n}\right)\sigma(A).
\]
\end{proposition}

\begin{proof}
By \Cref{thm: two set probability}, we have
\[
\E_{x}[\sigma_{x, r}(A \cap S_{x, r}) \mid x \in A] = \Pr_{x \cdot y = r}(y \in A \mid x \in A) \geq \sigma(A)^{(1 + |r|)/(1 - |r|)} - O_r\left(\frac{1}{n}\right).
\]
Split the expectation by conditioning on \( A_{\mathrm{good}} \) and its complement:
\begin{align*}
\E_{x}[\sigma_{x, r}(A \cap S_{x, r}) \mid x \in A]
&= \E[\sigma_{x, r}(A \cap S_{x, r}) \mid x \in A_{\mathrm{good}}] \Pr(x \in A_{\mathrm{good}} \mid x \in A) \\
&\quad + \E[\sigma_{x, r}(A \cap S_{x, r}) \mid x \notin A_{\mathrm{good}}] \Pr(x \notin A_{\mathrm{good}} \mid x \in A) \\
&\leq 1 \cdot \Pr(x \in A_{\mathrm{good}} \mid x \in A) + \frac{1}{2} \sigma(A)^{(1 + |r|)/(1 - |r|)}.
\end{align*}
Rearranging yields
\[
\Pr(x \in A_{\mathrm{good}} \mid x \in A) \geq \E_{x}[\sigma_{x, r}(A \cap S_{x, r}) \mid x \in A] - \frac{1}{2}\sigma(A)^{(1 + |r|)/(1 - |r|)},
\]
and thus
\[
\sigma(A_{\mathrm{good}}) = \Pr(x \in A_{\mathrm{good}} \mid x \in A) \cdot \sigma(A) \geq \frac{1}{2}\sigma(A)^{2/(1 - |r|)} - O_r\left(\frac{1}{n}\right)\sigma(A) \ . \qedhere
\]
\end{proof}
The following proposition shows the relation between inner product on $\sn$ and the normalized inner product on $S_{x,c}$:
\begin{proposition}
\label{prop: inner product relation}
Let \( x_1, x_2, x_3 \in \sn \) be such that \( \inner{x_1, x_2} = \inner{x_1, x_3} = c \) and \( \inner{x_2, x_3} = r \). For \( i = 2, 3 \), we can write
\[
x_i = c x_1 + \sqrt{1 - c^2} y_i,
\]
where \( y_i = \frac{x_i - c x_1}{\|x_i - c x_1\|_2} \) satisfies \( \inner{y_i, x_1} = 0 \) and
\[
\inner{y_2, y_3} = f_c(r),
\]
with \( f_c(r) = \frac{r - c^2}{1 - c^2} \).
\end{proposition}

\begin{proof}
A direct computation shows that \( y_i \) is orthogonal to \( x_1 \) and that \( \|x_i - c x_1\|_2 = \sqrt{1 - c^2} \). Thus,
\begin{align*}
\inner{y_2, y_3} &= \frac{\inner{x_2 - c x_1, x_3 - c x_1}}{1 - c^2} \\
&= \frac{\inner{x_2, x_3} - c \inner{x_1, x_3} - c \inner{x_1, x_2} + c^2 \|x_1\|_2^2}{1 - c^2} \\
&= \frac{r - c^2}{1 - c^2} \ . \qedhere
\end{align*}
\end{proof}

We now apply the above propositions to prove our main result using induction.

\begin{theorem}
\label{thm: density R configuration}
Let \( (v_1, \dots, v_k) \) be an inductive configuration with covariance matrix \( R \), and assume that
\[
\norm{v_k - v_{k-1}}_2 \neq \operatorname{diam}\left(\bigcap_{j < k-1} S_{v_j, r_j}\right).
\]

Then for any measurable set \( A \subseteq \mathbb{S}^{n-1} \) with \( \sigma(A) \geq \omega_R(n^{-\epsilon_R}) \), the probability that a uniformly random tuple \( (x_1, \dots, x_k) \in \Delta(n, R) \) lies entirely in \( A \) satisfies
\[
\Pr_{(x_1, \dots, x_k) \in \Delta(n, R)}(x_1 \in A, \dots, x_k \in A) \geq \Omega_R(\sigma(A)^{C_R}),
\]
as \( n \to \infty \). The constant \( C_R \) is given by
\[
C_R = \sum_{i=1}^{k-1} \frac{2}{1 - |c_i|} \prod_{j=1}^{i-1} \frac{1 + |c_j|}{1 - |c_j|},
\]
with
\[
c_1 = r_1, \quad c_i = (f_{c_{i-1}} \circ \cdots \circ f_{c_1})(r_i), \quad i = 2, \dots, k-1,
\]
and \( f_c(r) = \frac{r - c^2}{1 - c^2} \). The constant $\epsilon_R$ is given by
$$\epsilon_R=\prod_{i=1}^{k-1}\frac{1-|c_i|}{1+|c_i|}.$$
\end{theorem}

\begin{proof}
We proceed by induction on \( k \). The base case \( k = 2 \) follows directly from \Cref{thm: two set probability}. The condition \( v_2 - v_1 \neq \operatorname{diam}(\sn) \) ensures \( r = \inner{v_1, v_2} > -1 \) and the condition $\sigma(A)\geq \omega_R(n^{-\epsilon_R})$ ensures that the error term goes to $0$.

Assume the result holds for all configurations of size \( < k \), and let \( c_1 = r_1 \). Define
\[
A_{\mathrm{good}} = \left\{x \in \sn : \sigma_{x, c_1}(A \cap S_{x, c_1}) \geq \frac{1}{2} \sigma(A)^{(1 + |c_1|)/(1 - |c_1|)} \right\}.
\]
Then,
\begin{align*}
\Pr_{(x_1, \dots, x_k) \in \Delta(n, R)}(x_1, \dots, x_k \in A)
&= \Pr(x_1 \in A) \cdot \Pr(x_2, \dots, x_k \in A \mid x_1 \in A) \\
&\geq \Pr(x_1 \in A_{\mathrm{good}}) \cdot \Pr(x_2, \dots, x_k \in A \mid x_1 \in A_{\mathrm{good}}).
\end{align*}

Given \( x_1 \), we write \( x_i = c_1 x_1 + \sqrt{1 - c_1^2} y_i \) for \( i \geq 2 \), where \( y_i =\frac{x_i-r_1x_1}{\norm{x_i-r_1x_1}_2} \). Conditioned on $x_1$, according to \Cref{prop: inner product relation}, the induced configuration \( (y_2, \dots, y_k) \)  follows from a uniform distribution over \( \Delta(n - 1, R') \), where
\[
R' = R(f_{c_1}(r_2), \dots, f_{c_1}(r_k)).
\]
Given $x_1$, the event $x_2\in A,\cdots,x_k\in A$ is equal to $y_2\in A\cap S_{x_1,c_1},\cdots ,y_k\in A\cap S_{x_1,c_1}$. Conditioned on \( x_1 \in A_{\mathrm{good}} \), we have \( \sigma_{x_1, c_1}(A \cap S_{x_1, c_1}) \geq \frac{1}{2} \sigma(A)^{(1 + |c_1|)/(1 - |c_1|)}\geq \omega_{R'}(n^{-\epsilon_{R'}})\), so by the inductive hypothesis:
\[
\Pr(y_2, \dots, y_k \in A \cap S_{x_1, c_1}) \geq \Omega_R\left(\sigma(A)^{C_{R'} (1 + |c_1|)/(1 - |c_1|)}\right).
\]
with
$$C_{R'}=\sum_{i=2}^{k-1}\frac{2}{1-|c_i|}\prod_{j=2}^{i-1} {\frac{1+|c_j|}{1-|c_j|}}$$
where $c_2=f_{c_1}(r_2)$ and
    $$c_{i}=(f_{c_{i-1}}\circ \cdots \circ f_{c_2})(f_{c_1}(r_i))=(f_{c_{i-1}}\circ \cdots \circ f_{c_1})(r_i),\ i=3,\cdots,k-1$$
By \Cref{prop: A_good lower bound}, we also have
\[
\Pr(x_1 \in A_{\mathrm{good}}) \geq \Omega_R(\sigma(A)^{2/(1 - |c_1|)}).
\]
Multiplying the two bounds give
\[
\Pr(x_1, \dots, x_k \in A) \geq \Omega_R\left(\sigma(A)^{C_R}\right),
\]
where
\[
C_R = \frac{2}{1 - |c_1|} + \frac{1 + |c_1|}{1 - |c_1|} C_{R'} = \sum_{i = 1}^{k-1} \frac{2}{1 - |c_i|} \prod_{j = 1}^{i-1} \frac{1 + |c_j|}{1 - |c_j|} \ .\qedhere
\]
\end{proof}

\begin{remark}
\label{rem: diameter condition}
The condition 

\[
\norm{v_k - v_{k-1}}_2 \neq \operatorname{diam}\left(\bigcap_{j < k-1} S_{v_j, r_j}\right)
\]
is equivalent to \( c_{k-1} \neq -1 \), which is necessary for applying \Cref{thm: two set probability}. For a given configuration \( (v_1, \dots, v_k) \), the quantity \( c_i \) corresponds to the inner product between \( v_i \) and later vectors \( v_j \), after projecting to the $(n-i)$-subsphere \( \bigcap_{l < i} S_{v_l, r_l} \) and normalizing. Since \( v_1, \dots, v_k \) are distinct, we have \( -1 < c_i < 1 \) for all \( i \leq k - 2 \) and \( c_{k-1} < 1 \). The only forbidden case is \( c_{k-1} = -1 \), which occurs precisely when \( v_k - v_{k-1} \) equals the diameter of the intersection $\bigcap_{j < k-1} S_{v_j, r_j}$.
\end{remark}

A configuration is \textit{sphere Ramsey} if for any $c>0$, any measurable $c$-coloring of the sphere contains a monochromatic congruent copy of the configuration. Since any measurable $c$-coloring must contain a monochromatic set $A$ with $\sigma(A)\geq 1/c$, an immediate corollary to \Cref{thm: density R configuration} is that all inductive configurations (excluding the special case) are sphere Ramsey.
\begin{corollary}
    Let \( (v_1, \dots, v_k) \) be an inductive configuration with covariance matrix \( R \), and assume that
\[
\norm{v_k - v_{k-1}}_2 \neq \operatorname{diam}\left(\bigcap_{j < k-1} S_{v_j, r_j}\right).
\]
Then $R$ is sphere Ramsey (for measurable colorings).
\end{corollary}

A particularly important class of inductive configurations are \( k \)-simplices:
\[
\Delta_k(n, r) := \Delta(n, R(r, \dots, r)) = \{(x_1, \dots, x_k) \in (\sn)^k : \inner{x_i, x_j} = r,\ \forall i \neq j \}.
\]
In this case, the coefficients \( c_i \) admit the closed-form expression \( c_i = \frac{r}{1 + (i - 1)r} \). The condition \( c_{k-1} > -1 \) is equivalent to \( r > -\frac{1}{k - 1} \). We thus obtain the following corollary:

\begin{corollary}
\label{col: k simplex density}
Fix any \( r \in \left(-\frac{1}{k - 1}, 1\right) \). For any measurable set \( A \subseteq \mathbb{S}^{n-1} \) with \( \sigma(A) \geq \omega_{k,r}(n^{-\epsilon_{k,r}}) \), the probability that a uniformly random tuple \( (x_1, \dots, x_k) \in \Delta_k(n, r) \) lies entirely in \( A \) satisfies
\[
\Pr_{(x_1, \dots, x_k) \in \Delta_k(n, r)}(x_1 \in A, \dots, x_k \in A) \geq \Omega_{k,r}(\sigma(A)^{C_{k,r}}),
\]
as \( n \to \infty \). The constants \( C_{k,r},\ \epsilon_{k,r} \) are given by
\[
C_{k,r} = \sum_{i = 1}^{k - 1} \frac{2}{1 - |c_i|} \prod_{j = 1}^{i - 1} \frac{1 + |c_j|}{1 - |c_j|}
\qquad
\epsilon_{k,r}=\prod_{i=1}^{k-1}\frac{1-|c_i|}{1+|c_i|},
\]
where
\[
c_i = \frac{r}{1 + (i - 1)r}, \quad \text{for } i = 1, \dots, k - 1.
\]
\end{corollary}

\section{Open Problems}

Our density version of the Frankl–Rödl theorem on the sphere (\Cref{thm: two set probability}) requires the set $A \subseteq \mathbb{S}^{n-1}$ to have spherical measure $\sigma(A) \gg n^{-(1-|r|)/(1+|r|)}$ in order to guarantee a non-trivial lower bound. In contrast, the classical Frankl–Rödl theorem asserts that for any $-1 < r < 1$, there exists a constant $\epsilon = \epsilon(r) > 0$ such that
$$
\Pr_{x \cdot y = r}(x \in A, y \in A) > 0
$$
holds for all sets $A$ with $\sigma(A) > \Omega(\epsilon^n)$. When $r = 0$, our result in \Cref{thm: orthogonal simplex} provides a quantitative lower bound valid for all sets $A$ with $\sigma(A) \geq C\exp(-cn^{1/3})$, which is super-polynomially small but not exponentially small in $n$. A natural open problem is to extend our techniques and obtain explicit lower bounds on
$$
\Pr_{x \cdot y = r}(x \in A, y \in A)
$$
for all $-1 < r < 1$, even when $\sigma(A)$ is exponentially small in $n$.
\\

Another direction concerns the class of configurations we consider. Our density result applies to \textit{inductive configurations}, but Matoušek and Rödl \cite{matouvsek1995ramsey} showed that fix constant $c$ and for any configuration $P$ with circumradius less than 1, any $c$ coloring of $\sn$ contains a monochromatic copy of $P$. It remains open whether one can establish a \textit{density} version of this theorem—that is, to show
$$
\Pr_{(x_1, \ldots, x_k) \in P}(x_i \in A\ \forall i) > \epsilon(\sigma(A))
$$
for any set $A \subseteq \mathbb{S}^{n-1}$ with constant density $\sigma(A) > \epsilon(P)$ where $\epsilon(\sigma(A))$ is a lower bound on the probability in terms of $\sigma(A)$. 
\\

Finally, our current results can be almost extended to 3-point configurations. For example, using \Cref{thm: two set probability}, we can obtain bounds on
$$
\Pr_{\substack{x \cdot y = r_1 \\ x \cdot z = r_2}}(x \in A, y \in A, z \in A),
$$
but this does not yet yield control over 3-point configurations, which require bounding
$$
\Pr_{\substack{x \cdot y = r_1 \\ x \cdot z = r_2 \\ y \cdot z = r_3}}(x \in A, y \in A, z \in A).
$$
It is an open question whether a density theorem can be proved for general 3-point configurations on the sphere.

\bibliography{main}
\bibliographystyle{alpha}
\end{document}